\documentclass[12pt]{article}
\usepackage{hyperref}

\usepackage{standalone}%

\usepackage{amssymb,amsfonts,amscd,amsthm}
\usepackage[all,arc]{xy}
\usepackage{enumerate, bbm}
\usepackage{mathrsfs}
\usepackage{tikz}
\usepackage{mathtools}

\usepackage{color}%
\usepackage{graphicx}
\usepackage{enumerate}
\usepackage[left=0.9in,top=0.9in,right=0.9in,bottom=0.9in]{geometry}

\newtheorem{theorem}{Theorem}[section]
\newtheorem{corollary}[theorem]{Corollary}
\newtheorem{proposition}[theorem]{Proposition}
\newtheorem{lemma}[theorem]{Lemma}
\newtheorem{lem}[theorem]{Lemma}

\newtheorem{definition}[theorem]{Definition}

\newtheorem{problem}[theorem]{Problem}

\newcommand{\om}{\omega}

\newcommand{\half}{\frac{1}{2}}
\newcommand{\quart}{\frac{1}{4}}

\newcommand{\sm}{\setminus}
\newcommand{\sub}{\subseteq}

\DeclareMathOperator{\ex}{ex}
\DeclareMathOperator{\sat}{sat}

\renewcommand{\c}[1]{\mathcal{#1}}

\renewcommand{\c}[1]{\mathcal{#1}}

\newcommand{\f}[2]{\frac{#1}{#2}}
\newcommand{\floor}[1]{\lfloor #1\rfloor}

\title{Linear Bounds for Cycle-free Saturation Games}

\author{Sean English\footnote{Department of Mathematics, University of Illinois at Urbana-Champaign {\tt senglish@illinois.edu}.} \and Tom\'a\v{s} Masa\v{r}\'ik\footnote{Department of Applied Mathematics, Faculty of Mathematics and Physics, Charles University, Prague, Czech Republic \& Faculty of Mathematics, Informatics and Mechanics, University of Warsaw, Warsaw, Poland \& Department of Mathematics, Simon Fraser University, Burnaby, BC, Canada {\tt masarik@kam.mff.cuni.cz}} \and Grace McCourt\footnote{Department of Mathematics, University of Illinois at Urbana-Champaign {\tt mccourt4@illinois.edu}.} \and Erin Meger\footnote{Department of Mathematics and Statistics, Concordia University {\tt erin.meger@concorida.ca}.}  \and Michael S. Ross\footnote{Department of Mathematics, Iowa State University {\tt msross@iastate.edu}} \and Sam Spiro\footnote{Department\ of Mathematics, UC San Diego {\tt sspiro@ucsd.edu}.}}

\date{}

\begin{document}
	
	\maketitle
		\begin{abstract}
		Given a family of graphs $\mathcal{F}$, we define the $\mathcal{F}$-saturation game as follows.  Two players alternate adding edges to an initially empty graph on $n$ vertices, with the only constraint being that neither player can add an edge that creates a subgraph in $\mathcal{F}$.  The game ends when no more edges can be added to the graph.  One of the players wishes to end the game as quickly as possible, while the other wishes to prolong the game.  We let $\textrm{sat}_g(n,\mathcal{F})$ denote the number of edges that are in the final graph when both players play optimally.
		
		In general there are very few non-trivial bounds on the order of magnitude of $\textrm{sat}_g(n,\mathcal{F})$. In this work, we find collections of infinite families of cycles $\mathcal{C}$ such that $\textrm{sat}_g(n,\mathcal{C})$ has linear growth rate.
	\end{abstract}
	
	\section{Introduction}
	Given a family of graphs $\mathcal{F}$, we say a graph $G$ is $\mathcal{F}$-saturated if $G$ does not contain a subgraph isomorphic to any $F\in \mathcal{F}$, but adding any edge to $G$ creates a copy of some $F\in\mathcal{F}$. The study of $\mathcal{F}$-saturated graphs is one of the main topics of interest in extremal combinatorics. Particularly well studied quantities include the \textit{extremal number} $\ex(n,\c{F})$, which denotes the maximum number of edges in an $n$-vertex $\c{F}$-saturated graph, and the \textit{saturation number} $\sat(n,\c{F})$,  which denotes the minimum number of edges in an $n$ vertex $\c{F}$-saturated graph. If $\c{F}=\{F\}$, then we will often denote $\c{F}$ simply by $F$.
	
	For a family of graphs $\c{F}$, we define the $\c{F}$-saturation game as follows.  The game is played by two players, Max and Mini.  Starting with Max, the two players alternate adding edges to an empty graph on $n$ vertices, with the only constraint being that neither player can add an edge that creates a subgraph in $\mathcal{F}$. The game ends once the graph becomes $\mathcal{F}$-saturated.
	We define the $\c{F}$-game saturation number $\mathrm{sat}_g\left(n,\mathcal{F}\right)$ to be the number of edges in the graph at the end of the game when both players play optimally. 
  
  It is worth noting that
	\begin{equation}
	\sat(n,\mathcal{F})\leq \sat_g(n,\mathcal{F})\leq \ex(n,\mathcal{F}).
	\label{eq:basic}\end{equation}

  The first $\c{F}$-saturation game to be considered was the $C_3$-saturation game, which was introduced by F{\"u}redi, Reimer, and Seress \cite{FRS} based on another game proposed by Hajnal.  Despite being introduced nearly 30 years ago, very little is known about $\sat_g(n,C_3)$.  In their original paper,  F{\"u}redi, Reimer, and Seress proved that $\sat_g(n,C_3)\ge \half n\log(n)+o(n\log(n))$.  The only other non-trivial bound for $\sat_g(n,C_3)$ was obtained recently by  Bir{\'o}, Horn, and Wildstrom \cite{BHW} who showed that $\sat_g(n,C_3)\le \f{26}{121}n^2+o(n^2)$.
	
	The systematic study of saturation games was initiated by Carraher, Kinnersley, Reiniger, and West \cite{CKRW}. In particular, they obtained bounds for the game saturation number of paths and stars, and these results were refined by Lee and Riet \cite{LR}.  Hefetz, Krivelevich, Naor, and Stojakovi{\'c} \cite{HKNS} further generalized saturation games to avoiding other graph properties such as colorability, and in particular Keusch \cite{K} proved asymptotically tight bounds for the game where both players must keep the graph $4$-colorable. Saturation games have also been generalized to other combinatorial structures such as hypergraphs and directed graphs \cite{directedgraphs, Patkos}.

	In view of \eqref{eq:basic}, the problem of determining the order of magnitude of $\sat_g(n,\c{F})$ is trivial whenever $\ex(n,\c{F})$ and $\sat(n,\c{F})$ have the same order of magnitude, and this will often be the case if $\c{F}$ contains a tree.  Perhaps the simplest non-trivial case of this problem is to try and determine the order of magnitude of $\sat_g(n,\c{C})$ when $\c{C}$ is a family of cycles.
	
  A basic question in this direction that one could ask is: what families of cycles $\c{C}$ have quadratic game saturation number?  It is well known that $\ex(n,C_{2k})=o(n^2)$ for all $k$, so by~\eqref{eq:basic} a necessary condition to have $\sat_g(n,\c{C})=\Theta(n^2)$ is that $\c{C}$ consists only of odd cycles.  The last author~\cite{spiro} showed that a sufficient condition for a set of odd cycles $\c{C}$ to have quadratic game-saturation number is to have $C_3,C_5\in \c{C}$, in which case we have $\sat_g(n,\c{C})\ge \f{6}{25}n^2+o(n^2)$.  Carraher, Kinnersley, Reiniger, and West~\cite{CKRW} showed that $\sat_g(n,\c{C}^o)=\floor{n^2/4}$ where $\c{C}^o$ is the set of all odd cycles \cite{CKRW}, though the last author~\cite{spiro} showed that in general a set of odd cycles containing $C_3$ and $C_5$ need not have game saturation number asymptotic to $\quart n^2$.
	
	Similarly one could ask for necessary and sufficient conditions for a family of cycles $\c{C}$ to have linear game saturation number, and for this problem much less is known.  Let $\mathcal{C}_{\geq k}$ (respectively, $\mathcal{C}^o_{\geq k}$) denote the set of all cycles (respectively, all odd cycles) of length at least $k$. The following result of Erd\H{o}s and Gallai shows that a trivial condition to have $\sat_g(n,\c{C})=O(n)$ is for $\c{C}$ to contain every cycle which is at least as large as some cutoff value $k$.
  \begin{theorem}[{\cite[Theorem~2.7]{EG}}]\label{thm:EG}
		For all $n,k$ we  have \[\ex(n,\mathcal{C}_{\geq k})\le \half(k-1)(n-1).\]
	\end{theorem}

	Prior to this work, the only known non-trivial example of a set of cycles with linear game saturation number was the following.

  \begin{theorem}[{\cite[Theorem~1.4]{spiro}}]\label{thm:old}
		\[\sat_g(n,\mathcal{C}^o_{\geq 5})\le 2n-2.\]
	\end{theorem}
	
  The key idea in the proof of Theorem \ref{thm:old} is that Mini is able to play in the $\c{C}^o_{\ge  5}$-saturation game such that the graph stays $\mathcal{C}_{\geq 5}$-free throughout the game, so the result follows by Theorem~\ref{thm:EG}.  Our main goal is this paper is to generalize the approach used in Theorem~\ref{thm:old} to prove that $\sat_g(n,\c{C})=O(n)$ for many more families of cycles $\c{C}$.

  \subsection{Our Results}
	Our first result shows that if $\c{C}$ includes roughly half the cycles of length at least as large as some cutoff value $k$, then it has game saturation number which grows linearly with $n$.
	
	\begin{theorem}\label{thm:fin}
		Let $\mathcal{C}$ be a collection of cycles such that $C_3\not\in\mathcal{C}$, and such that there exists some $k\geq3$ so that for all $\ell\geq k$, either $C_\ell\in \mathcal{C}$ or $C_{\ell+1}\in\mathcal{C}$. Then
		\[
		\mathrm{sat}_g(n,\mathcal{C})\leq \half (3k-1)(n-1).
		\]
		
	\end{theorem}
	Applying this theorem with $2k+1$ in place of $k$ immediately gives the following.
	\begin{corollary}
		For $k\ge 2$,  \[\sat_g(n, \c{C}_{\ge 2k+1}^o )\le (3k+1)(n-1).\]
	\end{corollary}

Our next result applies to infinite families of cycles which are much sparser than those considered in Theorem \ref{thm:fin}. 

\begin{definition}\label{definition k dense}
	For $k\geq 5$, a family of cycles $\mathcal{C}$ is said to be \emph{$k$-dense} if the following properties hold: 
	
	\begin{enumerate}
		\item The cycles $C_k,C_{k+1}\in \mathcal{C}$, unless $k=5$ in which case we only require $C_5\in \mathcal{C}$,
		
		\item For all $\ell<k$, we have $C_{\ell}\notin \mathcal{C}$,
		
		\item For all $s \geq 3$ there exists $\ell$ with $s+2\le \ell\le 3+(k-2)(s-2)$ and $C_\ell\in \mathcal{C}$. 
	\end{enumerate}
\end{definition}
Roughly speaking, a family $\c{C}$ is $k$-dense if it contains no cycle up to length $k$, we have $C_k,C_{k+1}\in \c{C}$, and the gaps between consecutive cycle lengths in $\c{C}$ grow no faster than an exponential function in $k-2$.
	
	\begin{theorem}\label{thm:dense}
    If $\c{C}$ is a $k$-dense family of cycles, then
	\[
	\mathrm{sat}_g(n,\mathcal{C})\leq \half(k-1)(n-1)+1.
	\]
	\end{theorem}	\begin{corollary}
  Let $\c{C}=\bigcup_{r\ge 0} \{C_{3^r+4}\}=\{C_5,C_7,C_{13},C_{31},C_{85},\ldots\}$. Then 
		\[
		\sat_g(n,\c{C})\le 2n-1. 
		\]
	\end{corollary}
	
	Note that the gaps between the cycle lengths of this family $\c{C}$ grow exponentially large.  Moreover,  $\ex(n,\c{C})=\Theta(n^2)$ because $\c{C}$ consists only of odd cycles, so in theory its game saturation number could have been much larger than linear.
	
    We will prove Theorem~\ref{thm:dense} by showing that when $\c{C}$ is a $k$-dense family of cycles, either player can play in the $\c{C}$-saturation game so that the graph never contains a cycle of length at most $k$, and given this the result will essentially follow from Theorem~\ref{thm:EG}.    We will show (see Proposition~\ref{prop:opt}) that the bounds in the definition of $\c{C}$ being $k$-dense are close to best possible for such a strategy to exist.

  The rest of the paper is organized as follows.  In Section~\ref{S:basics} we establish some basic graph theory notation and facts.  In Section~\ref{S:finite} we  give a short proof of Theorem~\ref{thm:fin}.  The majority of the paper is dedicated to the proof of Theorem~\ref{thm:dense} in Section~\ref{sec:dense}. Finally, we close with some concluding remarks and open problems in Section~\ref{S:con}.

  \subsection{Preliminaries}\label{S:basics}

	Given a graph $G = (V,E)$, we say that a graph $G' = (V', E')$ is a \emph{subgraph} of $G$ if $V' \subseteq V$ and $E'\sub E$.  We say that $G'$ is an \emph{induced subgraph} of $G$ if $E' = \{uv \in E: u, v \in V'\}$.  For a set $S \subseteq V$, we define the \emph{induced subgraph on $S$} to be $G[S] := (S, E_S)$ where $E_S = \{uv\in E : u,v \in S\}$.
  We will often write $G+e$ and $G-e$ to denote the graph $G$ obtained by adding or deleting the edge $e$ respectively. Given two vertices $u,v\in V(G)$, a path from $u$ to $v$ is a \emph{$u\text{--}v$ geodesic} if the path is of length $d(u,v)$. A path that is a $u\text{--}v$ geodesic for some choice of $u$ and $v$ will just be called a \emph{geodesic}. The \emph{circumference} of a graph $G$ is the length of the longest cycle in $G$.   We say that a path is \textit{$v$-avoiding} if it does not contain the vertex $v$.
	
  Many of our proofs will rely heavily on the block structure of our underlying graph.  We recall that a \emph{block} $B$ of a graph $G$ is a maximal subgraph of $G$ which is either a $K_2$ or 2-connected, and that the edges of the blocks partition the edges of $G$. By a slight abuse of notation, we will often refer to blocks, vertex sets of blocks, and edge sets of blocks simply as ``blocks''.

  We recall some basic facts about blocks, and we refer the reader to Diestel \cite{D} for more information on blocks. Whitney's Theorem states that a graph on at least $3$ vertices is $2$-connected if and only if every pair of vertices is connected by two internally vertex-disjoint paths. Thus given a block $B$ on at least three vertices and $u, v \in B$, there exists a pair of internally disjoint paths from $u$ to $v$. A \emph{dominating vertex} in a block $B$ is a vertex adjacent to all other vertices in $B$. We say two blocks are \emph{adjacent} if their vertex sets have a nonempty intersection. Note that two blocks can overlap in at most one vertex, so two blocks being adjacent is equivalent to them sharing exactly one vertex.

    We now move onto some less standard terminology. Given a path $P=(u_1,u_2,\dots,u_\ell)$ with $\ell \ge 2$, note that there is a unique list of blocks $B_1,B_2,\dots,B_r$ and integers $1\leq s_1<s_2<\dots<s_r\leq \ell-1$ such that for all $1\leq j\leq r$ and $s_j\leq i<s_{j+1}$, the edge $u_iu_{i+1}\in B_j$. We will say $(B_1,B_2,\dots,B_r)$ is the \emph{block geodesic associated with $P$}. Note further that if $P_1$ and $P_2$ are two $u\text{--}v$ paths, then $P_1$ and $P_2$ are associated with the same block geodesic, which we will call the $ u\text{--}v $ block geodesic. Furthermore, if the $u\text{--}v$ block geodesic consists of $r$ blocks, we will say that the \emph{block distance from $u$ to $v$}, denoted $bd(u,v)$, is~$r$. In the $\c{C}$-saturation game, we let $G^{(t)}$ be the graph after the $t^{\textrm{th}}$ edge has been added in the game.  We let $G^{(0)}$ be the initial (empty) graph of the game, and we let $G^{(T)}$ denote the graph once it has become $\c{C}$-saturated.

	\section{Proof of Theorem~\ref{thm:fin}}\label{S:finite}
	
  We use the Erd\H{o}s-Gallai theorem, Theorem~\ref{thm:EG}, to give a relatively short proof of Theorem~\ref{thm:fin}.

  \begin{proof}[Proof of Theorem ~\ref{thm:fin}]
		First, we will show that we can play such that at the end of each of our turns the following hold for $G^{(t)}$:  
		\begin{itemize}
		    \item[(i)] Every block that is not a $K_2$ block contains a triangle, and 
		    \item[(ii)] There is at most one non-trivial path of $K_2$ blocks, with this path containing at most three edges. When it exists, we denote this path by $P$.
		\end{itemize}
    \begin{figure}[t]
  \begin{center}
  \includegraphics{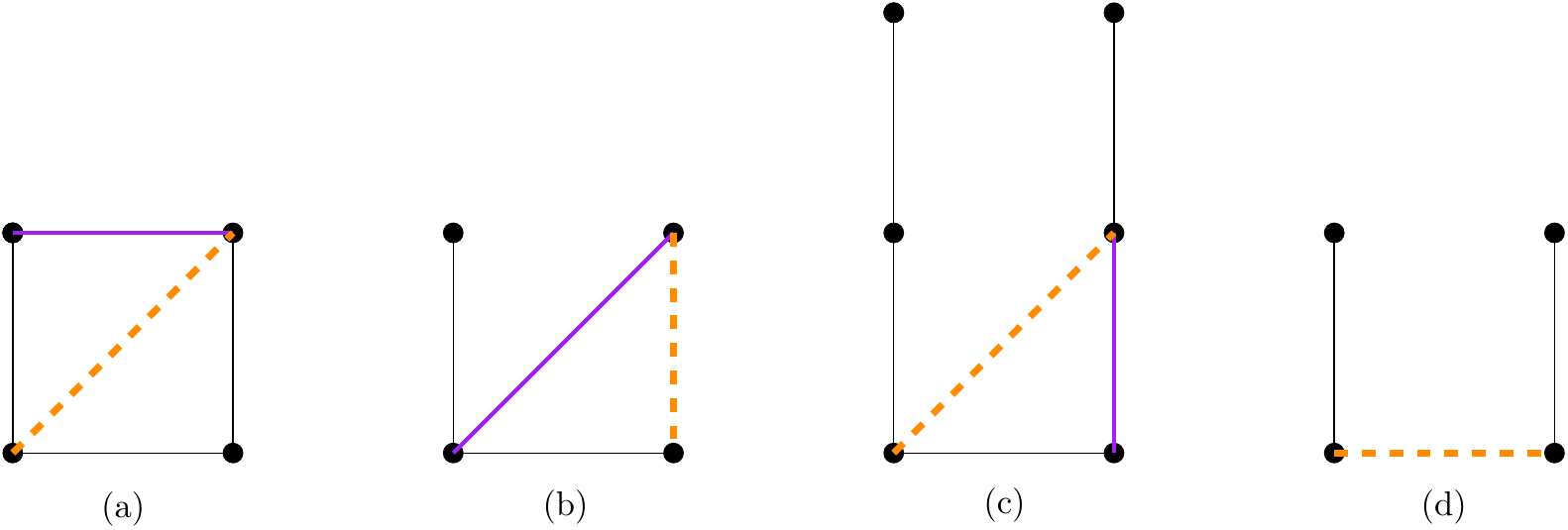}
\end{center}
  \caption{ An illustration of our moves in different situations depicted by dashed orange lines.
    The previous move of our opponent (if it is relevant) is highlighted by a thick purple line.
    Situation~(a) shows a case when property (i) is violated.
    Situations (b) and (c) show cases when  property (ii) is violated
    (in (b) due to multiple non-trivial paths of $K_2$ blocks, in (c) due to a long path of $K_2$ blocks).
    Situation (d) handles one possible case when properties (i) and (ii) are satisfied before our move.
    Note, that (d) is the only place when a path of $K_2$ blocks of length 3 is present after our move.
  }\label{fig:moves}
\end{figure}

  \noindent  We prove this claim inductively.  For the base cases, $G^{(1)}$ and $G^{(2)}$ have properties (i) and (ii), since there are not enough edges to negate the claim. 
    In the induction we describe moves that maintain properties (i) and (ii) in all possible situations; see also Figure~\ref{fig:moves}.
    Inductively assume $G^{(t)}$ has properties (i) and (ii). Then the only way a block without a triangle could form in $G^{(t+1)}$ would be for this block to consist of the new edge in $G^{(t+1)}$ and some number of $K_2$ blocks from $G^{(t)}$. Since by property (ii) there is only one non-trivial path of $K_2$ blocks, $P$, and as this has length at most $3$, the only blocks that could have formed in $G^{(t+1)}$ in this way are triangle blocks or (a unique) $C_4$ block.  If a unique $C_4$ block forms, we can add a diagonal to this block, which then contains a triangle and satisfies property (i), and it is not difficult to see that (ii) is satisfied. Thus we can maintain the desired conditions if $G^{(t+1)}$ contains a non $K_2$ block without a triangle.
	
		With this we can assume that every non $K_2$ block of $G^{(t+1)}$ contains a triangle, so property (i) is satisfied. If $G^{(t+1)}$ contains more than one non-trivial path of $K_2$ blocks, or a path of $K_2$ blocks of length $4$ or more, this must have been created by our opponent's last turn. In either case, we can choose an edge that connects our opponent's most recent edge to another edge within this $K_2$ block path. This will add a triangle and decrease the length of the longest path of $K_2$ blocks or add to the second non-trivial path of $K_2$ blocks such that it is no longer such a path. In either case, this satisfies property (ii) without violating property (i). 
		
		Next, we must handle the case that $G^{(t+1)}$ already satisfies properties (i) and (ii). If there exists a non-trivial path of $K_2$ blocks, then we can add an edge between two vertices on that path such that we create a triangle. The resulting graph satisfies both properties. Thus we may assume $G^{(t+1)}$ has no two adjacent $K_2$ blocks. Any move that does not create a forbidden cycle cannot create a path of $K_2$ blocks on more than three edges, so any move we make satisfies property (ii). Also, such a move can create at most one block $B$ which was not present in $G^{(t+1)}$. If $B$ is a $K_2$ block, then property (i) is satisfied. If the vertices of $B$ were contained in a single block in $G^{(t+1)}$, then that block must have contained a triangle, so property (i) is satisfied. Thus $B$ must contain the vertices from at least two blocks in $G^{(t+1)}$. If any of those blocks contained a triangle in $G^{(t+1)}$, then $B$ contains a triangle and property (i) is satisfied. Thus all the blocks from $G^{(t+1)}$ contained in $B$ are $K_2$ blocks. However, no two $K_2$ blocks are adjacent in $G^{(t+1)}$, and it is impossible to form a single block using only non-adjacent $K_2$ blocks and one additional edge. This completes the proof of the claim.

		Assume we play according to this strategy throughout the game, which implies that every block of $G^{(T)}$ is a $K_2$ block or contains a triangle.  We claim that this implies that $G^{(T)}$ is $\mathcal{C}_{\geq 3k}$-free. If some block $B$ of $G^{(T)}$ contained a cycle $C$ of length $\ell\geq 3k$ and a triangle $xyz$, then we will show that this block also contains a cycle of length $\ell'$ for some $\ell'\ge  \ell/3 +2$ which contains the three vertices in the triangle as consecutive vertices.   This implies that $G^{(T)}$ also  contains a cycle of length $\ell'-1$, which contradicts $G^{(T)}$ being $\c{C}$-free.
		
		Suppose the triangle $xyz$ intersects $C$ in two or three vertices, not necessarily sequentially along $C$. This yields two or three paths along $C$ between each of the vertices of $x,y,z$ that it intersects. One of these paths must have length at least $\frac{\ell}3$, call it $P_1$ with endpoints say $x,y$, without loss of generality. Then the cycle formed by $P_1$ and the edges $xz$, $zy$ give the desired cycle.

		Next, suppose that $C$ and $xyz$ have exactly one vertex in common, say $x$ without loss of generality. Since $B$ is 2-connected, there exists some shortest path $Q$ in $B-x$ from a vertex in $\{y,z\}$ to a vertex of $C-x$.  Since $Q$ is a shortest path, it contains at most one vertex from $\{y,z\}$ and $C-x$, say these vertices are $y,v$, which must be the endpoints of $Q$.  Then $x$ and $v$ split $C$ into two paths, one of which will have length at least $\frac{\ell}2\geq \frac{\ell}3$.  This long path, together with $Q$ and the edges $yz$ and $zx$, gives the desired cycle. 
		
		Finally, suppose that $C$ and $xyz$ are disjoint. Then since $B$ is $2$-connected, we can find two paths, disjoint from each other and internally-disjoint with both $xyz$ and $C$, connecting distinct vertices of the triangle to distinct vertices of the cycle. This again partitions $C$ into two paths, one of which will have length at least $\frac{\ell}3$, and using this together with the two paths from $xyz$ and the third vertex of $xyz$ gives the desired cycle.
		
		Thus $G^{(T)}$ is $\c{C}_{\geq 3k}$-free. Using Theorem~\ref{thm:EG} gives
		
		\[
		|E(G^{(T)})|\leq \frac{(3k-1)(n-1)}{2}.    \qedhere
		\]
	\end{proof}

\section{$k$-Dense Families of Cycles}\label{sec:dense}

In this section, we consider the cycle saturation game for families of cycles with sufficiently small of gaps between consecutive forbidden cycle lengths.

Let us briefly outline the strategy we use in the $\c{C}$-saturation game when $\c{C}$ is a $k$-dense family of cycles.  Ideally we would like to play such that at the end of each of our turns, $G^{(t)}$ is a connected graph where every block $B$ is dominated, i.e. $B$ contains a dominating vertex $r_B$, and moreover $r_B$ is the unique vertex of $B$ which is closest to some special vertex $h$.  Assuming this holds, if our opponent connects two vertices $u$ and $v$ at blocks distance $s$ from each other, then the path through their dominating vertices creates a cycle of length $s+1$.  We further require that each dominating vertex $r_B$ be the endpoint of a path of length $k-3$ in another block, which will allow us to extend this path through dominating vertices in such a way that we can actually get any cycle length between $s+1$ and $3+(k-2)(s-2)$, with at least one of these lengths forbidden by assumption of $\c{C}$ being $k$-dense.   Thus distant blocks can not be connected to one another.  This will imply that $G^{(t)}$ has small circumference, and hence relatively few edges.

It remains to describe how we can maintain this ideal structure (or at least something close to it).  If our opponent ever connects an isolated vertex $x$ to one of these blocks $B$, then we can try to make $x$ adjacent to the dominating vertex of $B$.  If this is impossible then it will turn out that $x$'s neighbor $y$ dominates the $xy$ block and $y$ is the end of a long path in $B$, so this new block maintains the desired properties.  The issue will be when our opponent connects two isolated vertices.  To maintain connectedness we are forced to make one of these vertices adjacent to $h$. By doing this repeatedly the opponent can create some non-desirable structures, but we can at least maintain that any that do appear are incident to $h$.

\subsection{$k$-Fantastic Graphs}
We now move onto our formal definitions. Throughout the remainder of this section, we will assume we are playing the $\mathcal{C}$-saturation game for a $k$-dense family of cycles $\mathcal{C}$. To define the structure that we wish to maintain in this game, we work with a graph $G$ with a specified vertex $h$.  We say a block $B$ in $G$ is \emph{rooted} if there exists a vertex $r_B$ that is a dominating vertex in $B$, and for which $d(r_B,h)<d(u,h)$ for all $u\in B - r_B$. When $B$ is rooted, we say $r_B$ is the \emph{root} of $B$ and that $r_B$ \emph{roots} the block $B$. Note that if $h\in B$ and $B$ is rooted, then $h$ is the root. Furthermore, note that if every block in $G$ has a root, then every vertex $v\in V(G)\backslash \{h\}$ is in exactly one block for which it is not the root, namely the first block in the $v\text{--}h$ block geodesic. If $v\neq h$ is a vertex that roots every block containing it except for one block $B$, then we will call this block the \emph{stem of $v$} and denote it by $B_v$.

We say a vertex $v$ in a rooted block $B$ is \emph{finished} if $v$ is the endpoint of a path of length at least $k-3$ in the induced subgraph $B-r_B$.  We say the block $B$ is \emph{finished} if all vertices in $B-r_B$ are finished. If a vertex or block is not finished, we call it \emph{unfinished}.
We say a block $B$ is \emph{nearly $h$-dominated} if $B$ contains $h$, and all but one vertex in $B$ is adjacent to $h$. In this case we define $r_B:=h$ (even though $h$ is technically not a root of the block). 

\begin{definition}\label{def:umbrella}
Two blocks $B_1$ and $B_2$ are together called an \emph{$h$-umbrella} if the following hold:

\begin{enumerate}
    \item The vertex $h$ dominates $B_1$,
    \item The block $B_2$ is a $K_2$ block,
    \item The blocks $B_1$ and $B_2$ intersect at a vertex $u\neq h$,
    \item For any other block $B$, if $B$ intersects $B_1$, then they intersect at $h$,
    \item The vertex in $B_2\sm B_1$ has degree 1, and we will refer to this vertex as the handle of the $h$-umbrella.
\end{enumerate}
We say that an $h$-umbrella is finished if the unique neighbor of the handle is finished in $B_1$, and unfinished otherwise.
\end{definition}

See Figure~\ref{fig:umbrella} for an illustration of an $h$-umbrella.  Note that if $B_1$ and $B_2$ constitute an $h$-umbrella, then $r_{B_1}=h$ and $r_{B_2}=u$.

We are now ready to state the definition of a $k$-fantastic graph, which is the main structural tool we will need in this section.

\begin{figure}
  \begin{center}
  \includegraphics{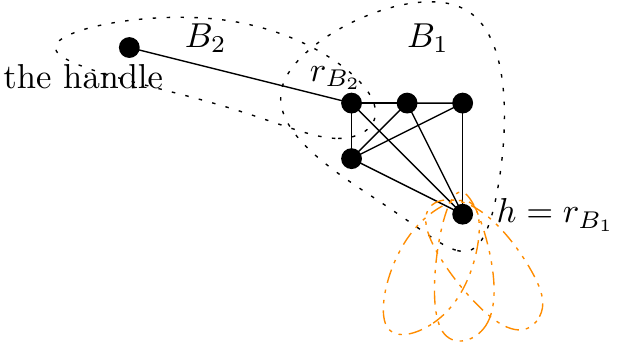}
\end{center}
  \caption{ An illustration of Definition~\ref{def:umbrella}. 
    The $h$-umbrella is in black.
    The vertex $r_{B_2}$ is the unique neighbor of the handle.
    Other blocks that intersects $B_1$ are depicted in orange.
  }\label{fig:umbrella}
\end{figure}

\begin{definition}\label{kfan}

Given a graph $G$ with specified vertex $h$, let $H$ be the subgraph induced by the set of vertices which are contained in a block which is nearly $h$-dominated or part of an unfinished $h$-umbrella, and if no such blocks exist set $H=\{h\}$. Let $F$ be the subgraph induced by the set consisting of $h$ and  all of the vertices contained in the  blocks of $G$ not contained in $H$.
We say $G$ is \emph{$k$-fantastic} whenever $F$ has the following properties: 

\begin{enumerate}[leftmargin=1cm]
	\item[{\normalfont Property 1.}] $F$ is connected,\label{property connected}
	\item[{\normalfont Property 2.}] Every block in $F$ is rooted,\label{property rooted}
	\item[{\normalfont Property 3.}] If a vertex $v\neq h$ is the root of some block in $F$, then $v$ is finished in its stem $B_v$. \label{property stem finished}
	\item[{\normalfont Property 4.}] Each vertex $v\neq h$ roots at most one unfinished block in $F$,\label{property one unfinished}
	\item[{\normalfont Property 5.}] The vertex $h$ is adjacent to at most one vertex of degree $1$ in $F$. \label{property one leaf}
	
\end{enumerate}
\end{definition}

We note that the graph $G$ consists of $H\cup F$ along with a set of isolated vertices. Note that in a slight abuse of notation, we will sometimes use $H$ and $F$ to refer to $V(H)$ and $V(F)$, respectively.

\subsection{Preliminary work towards $k$-Dense results}
We will ultimately show that for any $k$-dense family of cycles $\mathcal{C}$, either player can play the $\mathcal{C}$-saturation game such that the graph $G^{(t)}$ is $k$-fantastic at the end of their turns. To this end, throughout this subsection we let $G$ refer to a $k$-fantastic graph, and we consider $h$, $H$, $F$, and $I$ as defined above.  We also let $G'$ refer to any $\c{C}$-free graph with specified vertex $h$, which will usually be thought of as $G+uv$ for some edge $uv$.

\begin{lemma}\label{lemma short avoiding path}
	Let $G'$ be a $\mathcal{C}$-free graph with specified vertex $h$.
  If $B$ is a block of $G'$ that is either rooted or nearly $h$-dominated, then the longest $r_B$-avoiding path in $B$ is of length at most $k-3$.
\end{lemma}
\begin{proof}
	Assume to the contrary that $B$ contains an $r_B$-avoiding path $(v_0,v_1,\dots,v_{k-2})$. If the edges $v_0r_B$ and $v_{k-2}r_B$ are in $E(G')$, then $(r_B,v_0,v_1,\dots,v_{k-2},r_B)$ is a cycle of length $k$, which contradicts the fact that $G'$ is $\mathcal{C}$-free. This implies that $B$ must be nearly $h$-dominated, $r_B=h$, and exactly one of the edges $v_0h$ or $v_{k-2}h$ are not in $E(G')$. We assume without loss of generality $v_{k-2}h\not\in E(G')$. Since $B$ is non-trivial and $2$-connected, $v_{k-2}$ must have at least two neighbors in $B$, and specifically at least one neighbor $v$ with $v\neq v_{k-3}$. If $v\neq v_i$ for any $1\leq i\leq k-4$, then $(h,v_1,v_2\dots,v_{k-2},v,h)$ is a cycle of length $k$, a contradiction. If $v=v_i$ for some $1\leq i\leq k-4$, then $(h,v_{i+1},v_{i+2},\dots,v_{k-2},v_i,v_{i-1},\dots,v_0,h)$ is a cycle of length $k$ in $G'$, a contradiction. We conclude that for all rooted and nearly $h$-dominated blocks $B$, we have that $B$ does not contain an $r_B$-avoiding path of length $k-2$.
\end{proof}

\begin{lemma}\label{lemma no long cycles}
  The circumference of any $k$-fantastic graph is at most $k-1$.
\end{lemma}
\begin{proof}
  Let $G$ be a $k$-fantastic graph, and note that every block of a $k$-fantastic graph is either rooted or nearly $h$-dominated. If $G$ had a cycle of length $\ell\geq k$, then this cycle would have to be in some block $B$ because cycles are $2$-connected and blocks are maximal $2$-connected subgraphs. Thus, $B$ would necessarily contain an $r_B$-avoiding path of length $\ell-2> k-3$, a contradiction with Lemma~\ref{lemma short avoiding path}. So, the circumference of $G$ is at most $k-1$.
\end{proof}

We continue our discussion on the structural properties of $k$-fantastic graphs with the following two lemmas regarding which path lengths are attainable between specific vertices.

\begin{lemma}\label{lemma paths of every length from roots}
Let $G$ be $k$-fantastic. If a vertex $v$ is finished in a rooted block $B$ and $u \in B$, then for each $\ell$ with $d(u,v)\leq \ell\leq k-2$ there is a $u\text{--}v$ path of length $\ell$ in $B$.
\end{lemma}
\begin{proof}

  If $\ell=1$ the result is immediate, so assume $\ell\geq 2$. Since $v$ is finished there exists an $r_B$-avoiding path $(v,v_1,\dots,v_{k-3})$ in $B$. If $u=r_B$ then the path $(v,v_1,\ldots,v_{\ell-1},r_B)$ works.  If $u\ne r_B$ and $u\ne v_i$ for any $i<\ell$, then $(v,v_1,\dots,v_{\ell-2},r_B,u)$ is a path of length $\ell$.  If $u=v_i$ with $2\leq i<\ell$, then we use  $(v,v_1,\dots,v_{i-1},r_B,v_{\ell-1},v_{\ell-2},\dots,v_{i+1},u)$. In each case we find a path of length $\ell$, thus completing the proof.
\end{proof}

 Recall that we denote the block distance from vertex $u$ to vertex $v$ as $bd(u,v)$, a vertex $v$ can only have one stem, and that $v$ roots all other blocks containing it. 

\begin{lemma}\label{lemma path lengths between F and H}
	Let $G$ be $k$-fantastic, and let $u,v\in F$ with $u\ne h$ be such that $s:=bd(u,v) \geq 2$. Then there is a $u\text{--}v$ path of length $\ell$ in $G$ for every $\ell$ with \[s+1\leq \ell \leq 2+(k-2)(s-2).\]
	Moreover, if $v=h$ we can do this for all $\ell$ with
	\[s\leq \ell \leq 1+(k-2)(s-1).\]
\end{lemma}

\begin{proof}
Let $(B_1,\dots,B_s)$ be the $u\text{--}v$ block geodesic, and observe that these blocks are all rooted since $u,v\in F$.  For $1\leq i\leq s-1$, let $v_i$ be the vertex in both $B_i$ and $B_{i+1}$, and note that e.g. $u\ne v_1$ as otherwise $(B_2,\ldots,B_s)$ would be a shorter path of blocks from $u$ to $v$.  	We will first prove, regardless of if $v=h$ or $v\neq h$, that for all $d(v_i,v_{i+1})\le k_i\le k-2$ there is a $u\text{--}v$ path of length \[d(u,v_1)+\sum_{i=2}^{s-1} k_i+d(v_{s-1},v),\]
and moreover the path ends in a $v_{s-1}\text{--}v$ geodesic.
	
Indeed, starting from $u$ we transverse a $u\text{--}v_1$ geodesic.  If $v_1\ne r_{B_2}$, then by definition $B_{v_1}=B_2$, i.e. $B_2$ is the unique block containing $v_1$ which $v_1$ is not the root of.  The uniqueness of this block implies $v_1=r_{B_1}$, and by Property~\ref{property stem finished} this means that $v_1$ is finished in $B_2$.  Thus we can apply Lemma~\ref{lemma paths of every length from roots} to find a path of length $k_2$ from $v_1$ to $v_2$ in $B_2$.  If $v_1=r_{B_2}$, then $v_2\ne r_{B_2}$ and a symmetric argument gives the same conclusion.  Continuing in this manner, Lemma \ref{lemma paths of every length from roots} gives for each $3\leq i\leq s-1$ a path of length $k_i$ from $v_{i-1}$ to $v_{i}$ in $B_i$, and once we reach $v_{s-1}$ we can traverse a $v_{s-1}\text{--}v$ geodesic to complete the path.  This proves the claim.
	
By choosing the appropriate values for the $k_i$'s, we can find a $u\text{--}v$ path of length $\ell$ for every $\ell$ with \[d(u,v)\leq \ell \leq d(u,v_1)+d(v_{s-1},v)+(k-2)(s-2).\]  To finish the proof for the $v\ne h$ case, it suffices to show $d(u,v_1)+d(v_{s-1},v)\geq 2$ and $d(u,v)\leq s+1$. The first part is immediate, and the second part follows by considering the path $(u,r_{B_1},\ldots,r_{B_s},v)$ after deleting duplicated vertices if, say, $r_{B_i}=r_{B_{i+1}}$ for some $i$.

In the case where $v=h$, we can again apply the claim to find $u\text{--}h$ paths of length $\ell$ for all $\ell$ with $d(u,h)\leq \ell \leq d(u,v_1)+d(v_{s-1},h)+(k-2)(s-2)$ and which end in the edge $v_{s-1}h$. Note that $v_{s-1}$  does not root $B_s$ (since $h\in B_s$), and thus it must root $B_{s-1}$ so it will be be finished in $B_s$. By Lemma \ref{lemma paths of every length from roots} we can replace the edge $v_{s-1}h$ with a path of length $k_s$ for any $1\leq k_s\leq k-2$, thus allowing us to build a $u\text{--}h$ path of every length $\ell$ with $d(u,h)\leq \ell \leq d(u,v_1)+(k-2)(s-1)$. Since $d(u,h)=s$ and $d(u,v_1)\geq 1$, the result follows.

\end{proof}

We show that we can always find long paths between pairs of vertices in adjacent blocks, except for one exceptional case.

\begin{lemma}\label{lemma block distance 2 path of length k-1}
	Let $G$ be $k$-fantastic. Let $u,v\in F$ with $bd(u,v)=2$, and let $(B_1,B_2)$ be the $u\text{--}v$ block geodesic. Then there is a $u\text{--}v$ path of length $k-1$ in $G$, unless $h$ roots $B_1$ and $B_2$ and both $u$ and $v$ are unfinished.
\end{lemma}

\begin{proof}
  Let $x\in B_1\cap B_2$.  Note that $G$ being $k$-fantastic implies that $x$ roots at least one of these blocks, and that $x\ne u,v$ due to $u$ and $v$ having block distance $2$.  First consider the case that $x$ roots both $B_1$ and $B_2$.  We claim that either $u$ or $v$ are finished. If $x=h$, then this is true by assumption, and if $x\neq h$, then one of the blocks $B_1$ or $B_2$ must be finished by Property \ref{property one unfinished}, so one of $u$ or $v$ must be finished, and we assume without loss of generality that $u$ is finished in $B_1$. Then by Lemma \ref{lemma paths of every length from roots}, there is a $u\text{--}x$ path of length $k-2$ which can be extended using the edge $xv$ to a path of length $k-1$ that ends at $v$, so we are done in this case.
 
Now assume $x$ roots only one of the blocks, say without loss of generality $B_2$. Since $B_1$ is the stem of $x$, $x$ is finished in $B_1$ by Property \ref{property stem finished}, so by Lemma \ref{lemma paths of every length from roots} there is a path of length $k-2$ from $u$ to $x$, which can be extended to a $u\text{--}v$ path of length $k-1$ using the edge $xv$. Thus in all cases we are done.
\end{proof}

The next two lemmas will help us in situations in which our opponent plays an edge that is incident with a nearly $h$-dominated block.

\begin{lemma}\label{lemma H dominate okay}
Let $G'$ be a $\mathcal{C}$-free graph with specified vertex $h$. If $B$ is a nearly $h$-dominated block and $u\in B$ is the vertex not adjacent to $h$, then $G'+uh$ is $\mathcal{C}$-free.
\end{lemma}
\begin{proof}
	If adding the edge $uh$ creates a cycle in $\mathcal{C}$, this cycle would be contained in the vertices of $B$, which implies that there must be a $u\text{--}h$ path of length at least $k-1$ in $B$, and hence an $h$-avoiding path of length $k-2$ in $B$, but this contradicts Lemma \ref{lemma short avoiding path}.
\end{proof}

We remind the reader that only unfinished $H$-umbrellas are in $H$, which in particular means the unique neighbor of the handle in an $h$-umbrella of $H$ must be unfinished. We also recall for $k\geq 6$, we require that both $C_k$ and $C_{k+1}$ are in $\mathcal{C}$ for $\mathcal{C}$ to be $k$-dense (but for $k=5$ we do not require $C_6\in \mathcal{C}$).

\begin{lemma}\label{lemma mini monster I dealing with H}
	Let $G$ be $k$-fantastic for $k\ge 6$. Let $u,v\in H-h$ be distinct vertices such that one of the following holds:
	\begin{enumerate}
		\item $u$ and $v$ are in distinct nearly $h$-dominated blocks,
		\item $u$ and $v$ are both handles in distinct $h$-umbrellas, or
		\item One of $u, v$ is in a nearly $h$-dominated block while the other is the handle in an $h$-umbrella.
	\end{enumerate}
	Let $B$ be the block containing the edge $uv$ in $G':=G+uv$, and let $a,b\in B-h$ be the vertices such that $ah,bh\not\in E(G')$. If $G'$ is $\mathcal{C}$-free, then both $G'+ah$ and $G'+bh$ are also $\mathcal{C}$-free.
\end{lemma}
\begin{proof}
	Towards a contradiction, we may assume $G'+ah$ contains a cycle from $\mathcal{C}$, so there is an $a\text{--}h$ path in $B$ of length $\ell$ with $\ell\geq k-1$ such that $C_{\ell+1}\in \mathcal{C}$, say the path $P=(v_0:=a,v_1,v_2,\dots,v_{\ell-1},h)$.  We can assume without loss of generality that the vertices $a$ and $u$ are both in the same nearly $h$-dominated block or $h$-umbrella in $G$ (and consequently the same is true for $b$ and $v$). 
	
	We claim that this path contains the vertices $u$ and $v$ in that order. If not, then $P$ is either completely contained inside a nearly $h$-dominated block, which contradicts Lemma \ref{lemma H dominate okay}, or $P$ is completely contained inside an $h$-umbrella in $H$. By the definition of $H$, the neighbor of $a$ in this $h$-umbrella $v_1$ must be unfinished, but $(v_1,v_2,\dots,v_{\ell-1})$ is an $h$-avoiding path of length $\ell-2\geq k-3$, giving us a contradiction and proving the claim. 
	
  Let $A_1$ denote the sets of vertices in the nearly $h$-dominated block or $h$-umbrella in $G$ containing $a$ and $u$, and let $A_2$ denote the set of vertices in the nearly $h$-dominated block or $h$-umbrella in $G$ containing $b$ and $v$. Then $u=v_i$ and $v=v_{i+1}$ for some $0\leq i\leq \ell-1$, and note that $\{v_j\mid 0\leq j\leq i\}\subseteq A_1$ while $\{v_j\mid i+1\leq j\leq \ell-1\}\subseteq A_2$.
	
  We now claim that we can find a cycle of length $\ell+2$ in $G'$ of the form $(h,x,a=v_0,v_1,\dots,v_{\ell-1},h)$ for some vertex $x\in A_1$. Indeed, if $A_1$ is an $h$-umbrella, we can choose $x$ to be the unique neighbor of $a$ in $A_1$, then by the definition of $u$, we have that $u=a$, so $(h,x,a=u=v_0,v=v_1,v_2,\dots,v_{\ell-1},h)$ is such a cycle. Thus we may assume that $A_1$ is a nearly $h$-dominated block. Since $A_1$ is $2$-connected and non-trivial, $a$ has a neighbor $x$ that is not $v_1$ (nor $h$ since $ah\notin E(G')$). We consider two cases based on $x$.
	
	\textbf{Case 1:} The vertex $x=v_j$ for some $0\leq j\leq \ell-1$. Note that $2\leq j\leq i$ since $x\in A_1$. This implies that $v_{j-1}\in A_1 - a - h$, and so $hv_{j-1}$ is an edge. Then $(h,v_{j-1},v_{j-2},\dots,v_0=a,x=v_j,v_{j+1},\dots,v_{\ell-1},h)$ is a $C_{\ell+1}$ in $G'$, a contradiction to $C_{\ell+1}\in \c{C}$.
	
	\textbf{Case 2:} The vertex $x\neq v_j$ for any $0\leq j\leq \ell-1$. Since $x\in A_1 - a - h$, $hx\in E(G)$. Then $(h,x,a=v_0,v_1,\dots,v_{\ell-1},h)$ is a $C_{\ell+2}$ in $G'$ with $h$ and $a$ at distance $2$ along the cycle.
	
	Thus, we have exhibited a cycle $(h,x,a=u=v_0,v=v_1,v_2,\dots,v_{\ell-1},h)$ as claimed. This implies that $C_{\ell+2}\not\in \mathcal{C}$. Since $k\geq 6$, we have $C_k,C_{k+1}\in \mathcal{C}$.  We conclude $\ell\geq k$ and that $v_{k-2}$ and $v_{k-1}$ are defined. We now show that $G'$ contains either a $C_k$ or a $C_{k+1}$, which will give us a contradiction. Indeed, since $v_{k-2},v_{k-1}\neq a=v_0$, we have that $h$ is adjacent to at least one vertex $v_j$ with $j\in\{k-1,k-2\}$, so $(h,x,v_0,v_1,\dots,v_j,h)$ is a $C_{j+2}$ in $G$, where $j+2\in \{k,k+1\}$, proving the claim and completing the proof.
\end{proof} 

Our next lemma characterizes what moves within the cycle saturation game will leave the graph $k$-fantastic.  For the rest of the section, we refer to a \textit{legal move} as an allowable move in the $\c{C}$-saturation game for the implied family $\c{C}$ which is $k$-dense.

\begin{lemma}\label{lemma edges in F do not hurt}
  If $G$ is $k$-fantastic, and $u,v\in F$ are vertices such that $uv$ is a legal move, then $G+uv$ is also $k$-fantastic.  Further, $G+uv$ does not contain any nearly $h$-dominated blocks which were not in $G$.
\end{lemma}

\begin{proof}
	Note that adding an edge within a block does not interfere with any of the properties of being $k$-fantastic nor create a nearly $h$-dominated block. Thus we can assume $bd(u,v)=s\geq 2$. If $s \geq 3$, by Lemma \ref{lemma path lengths between F and H} there is a path of length $\ell$ for every $\ell$ with $s+1\leq \ell\leq (k-2)(s-2)+2$, which implies that $uv$ would complete a cycle in $\mathcal{C}$, a contradiction. Thus, we may assume $s=2$.
	
	Since $uv$ is a legal move there is no $u\text{--}v$ path of length $k-1$. By Lemma \ref{lemma block distance 2 path of length k-1}, both $u$ and $v$ are adjacent to $h$. Then $h$ still dominates the block resulting from adding $u$ and $v$ (in particular meaning it is not nearly $h$-dominated), and all the properties of being $k$-fantastic are retained as desired.
\end{proof}

The following lemma will allow us to focus our attention only on those cases in which our opponent makes a move that results in a graph that is not $k$-fantastic.

\begin{lemma}\label{lemma fantastic retained}
	If $G$ is $k$-fantastic but not $\mathcal{C}$-saturated, then there exists a legal move $uv$ such that $G+uv$ is $k$-fantastic.  Moreover, if $G$ has at most one nearly $h$-dominated block, then $uv$ can be chosen so that $G+uv$ has no nearly $h$-dominated blocks.
\end{lemma}

\begin{proof}
	First assume $G$ contains a nearly $h$-dominated block with $x$ the vertex in this block not adjacent to $h$.  By Lemma \ref{lemma H dominate okay} we can add $hx$, which creates a rooted block with root $h$.  This makes $G$ $k$-fantastic with no nearly $h$-dominated blocks.  Thus we may assume that $G$ contains no nearly $h$-dominated blocks.
	
	Assume $G$ has an isolated vertex $x$. If $h$ is adjacent to a vertex of degree $1$, say $y$, then  $xy$ is a legal move making an $h$-umbrella. Otherwise $hx$ is a legal move. Thus we may assume $G$ contains no isolated vertices.
	
	Suppose $G$ contains an $h$-umbrella in $H$, say with handle $y$ and $x$ its unique neighbor, and recall that $x$ cannot be finished. Thus $yh$ can be added without creating a forbidden cycle or nearly $h$-dominated block. We can then assume that $G$ contains no $h$-umbrellas in $H$, and consequently that $H=\{h\}$ is trivial.
	
  Since $H$ is trivial and there are no isolated vertices, we must have $G=F$.  By hypothesis there exists a legal move involving two vertices of $F$, and by Lemma \ref{lemma edges in F do not hurt} and any such move leaves the graph $k$-fantastic, so we are done.
\end{proof}
\subsection{Main Results for $k$-Dense Families}

We are now ready to prove our main structural result for this section.  

\begin{proposition}\label{proposition mini monster II all the cases}
Let $\c{C}$ be a $k$-dense set of cycles for some $k\ge 5$.  Then either player can play the $\mathcal{C}$-saturation game such that at the end of each of their turns, the graph is $k$-fantastic. Moreover, if $k=5$, then that player can guarantee that the graph contains no nearly $h$-dominated blocks at the end of each of their turns.

\end{proposition}

\begin{proof}
	
	Note that $G^{(0)}$ and $G^{(1)}$ are both trivially $k$-fantastic and do not contain nearly $h$-dominated blocks. Now let us assume that $G^{(t)}$ is $k$-fantastic for some $t\geq 0$, and if $k=5$ that further $G^{(t)}$ contains no nearly $h$-dominated blocks. We will show that we can play such that $G^{(t+2)}$ is $k$-fantastic, unless $G^{(t+1)}$ is already $\mathcal{C}$-saturated. In the analysis that follows, when $k=5$ we will not verify that our own move does not create a new nearly $h$-dominated block, but it is easy to verify that the only time our strategy has our move creating such a block is when $k\ge 6$ in Case 5c.
	
	We consider cases based on the edge added at time $t+1$ which we denote by $e=uv$.  We also let $I$ denote the set of isolated vertices of $G^{(t)}$.
	
	\textbf{Case 0:} $G^{(t+1)}$ is $k$-fantastic.  In this case we may apply Lemma \ref{lemma fantastic retained}, and in particular this leaves it so that $G^{(t+2)}$ has no nearly $h$-dominated blocks when $k=5$. Note that by Lemma \ref{lemma edges in F do not hurt} this handles the case that $u,v \in F$.

  \textbf{Case 1:} $u,v\in I$. We play the edge $uh$, which creates an $h$-umbrella in $H$ and maintains $G^{(t+2)}$ being $k$-fantastic.

\textbf{Case 2:} $u=h$. First note that we do not need to consider the case when $v\in F$ since $h\in F$ as well.

\underline{Case 2a:} $v\in H-h$. If $hv$ is contained inside a nearly $h$-dominated block, then this block becomes rooted with root $h$, so $G^{(t+1)}$ is $k$-fantastic and we are in Case 0. The only other possibility is that $v$ is the handle of an $h$-umbrella since all other vertices in $H$ are adjacent to $h$. As such, adding the edge $hv$ causes this $h$-umbrella to become a single block which is rooted with root $h$, so the graph is still $k$-fantastic and we are in Case 0.

\underline{Case 2b:} $v \in I$. If $h$ was not adjacent to a degree $1$ vertex in $G$, then we are still $k$-fantastic and in Case 0. Otherwise, by Property~\ref{property one leaf} there is exactly one other degree $1$ vertex $x$ adjacent to $h$, and adding the edge $vx$ creates a block rooted at $h$, leaving the graph $k$-fantastic.

\textbf{Case 3:} $u \in F-h$ and $v \in I$.  This adds an unfinished $K_2$ block rooted at $u$. We only consider the cases where $u$ is the root of another unfinished block (which violates Property \ref{property one unfinished}), and the case where $u$ is an unfinished vertex (which violates Property \ref{property stem finished}). In any other case, $G^{(t+1)}$ remains $k$-fantastic since the Proprieties \ref{property stem finished} and \ref{property one unfinished} cannot be affected by a new unfinished $K_2$ block, and thus we are in Case 0.

\underline{Case 3a:} $u$ is the root of an unfinished block $B$. We add an edge from $v$ to some unfinished vertex $x\in B$. This does not create any cycle of length at least $k$ since $x$ was unfinished, so this is a legal move. Now $u$ is only adjacent to at most one unfinished block again, and Property \ref{property one unfinished} holds. Since $v$ is adjacent to the root of this block, Property \ref{property rooted} holds as well.

\underline{Case 3b:} $u$ is an unfinished vertex in some block $B$. We add the edge $vr_B$. This does not create a cycle of length at least $k$ since $u$ was unfinished, so this is a legal move. The resulting block is rooted with root $r_B$, so properties \ref{property rooted} and \ref{property stem finished} hold in $G^{(t+2)}$. Thus, $G^{(t+2)}$ is $k$-fantastic.

\textbf{Case 4:} $u \in H-h$ and $v \in I$.

\underline{Case 4a:} $u$ is in a nearly $h$-dominated block $B$. We add the edge containing $h$ missing from $B$, which is legal by Lemma \ref{lemma H dominate okay}. This creates an $h$-umbrella, and thus $G^{(t+2)}$ is $k$-fantastic.

\underline{Case 4b:} $u$ is in an $h$-umbrella consisting of blocks $B_1$ and $B_2$ with $h\in B_1$. Let $y$ be the handle and $x$ its unique neighbor (possibly with $x=u$ or $y=u$). Then by the definition of $H$, $x$ is unfinished. Then we can add the edge $yh$ as this does not create a cycle in $\mathcal{C}$ due to the fact that $x$ is unfinished. Then $\tilde{B}_1:=B_1\cup B_2$ becomes a block rooted at $h$, and so if $\tilde{B}_2$ is the block containing $uv$, $\tilde{B}_1$ and $\tilde{B}_2$ constitute an $h$-umbrella, and thus $G^{(t+2)}$ is $k$-fantastic.

\textbf{Case 5:} $u,v \in H-h$.

\underline{Case 5a:} $u$ and $v$ are in the same nearly $h$-dominated block or the same $h$-umbrella in $H$. If $u$ and $v$ are both in a nearly $h$-dominated block, then this block remains nearly $h$-dominated. If $u$ and $v$ are both in the same $h$-umbrella, this either remains an $h$-umbrella or if either $u$ or $v$ was the handle in $G$, this becomes a nearly $h$-dominated block. In either case $G^{(t+1)}$ is $k$-fantastic, so we are in Case 0.

\underline{Case 5b:} $u$ is in an $h$-umbrella and is not the handle. By Case 5a, we may assume that $v$ is not in the same $h$-umbrella as $u$. If $v$ is in a nearly $h$-dominated block or $v$ is the handle of an $h$-umbrella, then the block containing $uv$ in $G^{(t+1)}$ is nearly $h$-dominated, so by Lemma \ref{lemma H dominate okay}, we can add the edge to turn this block into a rooted block with root $h$, which results in an $h$-umbrella so $G^{(t+2)}$ is $k$-fantastic.

It remains to consider when $v$ is also in an $h$-umbrella but not the handle. Note that the addition of the edge $uv$ forms a block rooted at $h$, adjacent to two rooted $K_2$ blocks in $G^{(t+1)}$. Let $x$ and $y$ be the handles of the original $h$-umbrellas containing $u$ and $v$ respectively, and let $x'$ and $y'$ be the neighbors of $x$ and $y$. If either $x'$ or $y'$ are unfinished in $G^{(t+1)}$, then we can add the edge $xh$ or $yh$ creating an $h$-umbrella and leaving $G^{(t+2)}$ $k$-fantastic. If both $x'$ and $y'$ are finished in $G^{(t+1)}$, then all these blocks are in $F$, so  $G^{(t+1)}$ is $k$-fantastic and we are in Case 0.

\underline{Case 5c:} The conditions of Case 5a and Case 5b are not met. Then one of the following holds:
\begin{enumerate}
	\item $u$ and $v$ are in distinct nearly $h$-dominated blocks,
	\item $u$ and $v$ are both handles of distinct $h$-umbrellas, or
	\item One of $u,v$ is in a nearly $h$-dominated block while the other is a handle of an $h$-umbrella.
\end{enumerate}
Note that if $k=5$, then we do not have nearly $h$-dominated blocks, and $u$ and $v$ cannot both be handles of $h$-umbrellas as $uv$ would create a $C_5$, which is forbidden. Hence, we may assume $k\geq 6$. Let $B$ be the block containing $uv$ in $G^{(t+1)}$, and note that $B$ is only adjacent to other blocks at $h$. By Lemma \ref{lemma mini monster I dealing with H}, in all cases we can add an edge containing $h$ and some other vertex of $B$, which will turn $B$ into a nearly $h$-dominated block, leaving $G^{(t+2)}$ $k$-fantastic. 
Since this case only happens with $k\geq 6$, we do not create $h$-dominated blocks when $k=5$, as desired.

\textbf{Case 6:} $u \in F-h$ and $v \in H-h$.

\underline{Case 6a:} The block distance $s := bd(u,h) \geq 2$. By Lemma \ref{lemma path lengths between F and H} we have a $u\text{--}h$ path of length $\ell$ for every $s\leq \ell \leq (k-2)(s-1)+1$, and thus $G^{(t+1)}$ contains a $C_{\ell'}$ for every $\ell'$ satisfying 
\[
s+d(h,v)+1\leq \ell'\leq (k-2)(s-1)+1+d(h,v)+1.
\] 
If $s=2$, then since $1\leq d(h,v)\leq 2$, we have that $G^{(t+1)}$ contains cycles of every length $\ell'$ with $5\leq \ell'\leq k+1$. In particular it contains $C_k$, a contradiction, so we may assume $s\geq 3$. If $d(h,v)=1$, then we have cycles of length $\ell'$ for all $s+2\leq \ell'\leq 3+(k-2)(s-2)<3+(k-2)(s-1)$, contradicting the fact that $\mathcal{C}$ must contain one such cycle. If $d(h,v)=2$, then let $s':=s+1$, and note that we have cycles of all lengths $\ell'$ with $s'+2\leq \ell'\leq (k-2)(s'-2)+3<(k-2)(s'-2)+4$, again a contradiction to the definition of $\mathcal{C}$ being $k$-dense.

\underline{Case 6b:} The block distance $s:= bd(u,h)=1$. Since $h$ roots every block contained in $F$ that $h$ is in, $uh\in E(G)$. Let $B$ be the block containing the edge $uh$. If $v$ is in a nearly $h$-dominated block or is the handle of an $h$-umbrella, then the addition of $uv$ just creates a nearly $h$-dominated block, leaving $G^{(t+1)}$ $k$-fantastic. If $v$ is in an $h$-umbrella but not the handle, then the addition of $uv$ creates a larger $h$-umbrella, which again leaves $G^{(t+1)}$ $k$-fantastic. Thus we are in Case 0.
\end{proof}

We can now prove our main result of this section.

\begin{proof}[Proof of Theorem~\ref{thm:dense}]
	By Proposition \ref{proposition mini monster II all the cases}, Mini can play such that at the end of each of her turns, $G^{(t)}$ is $k$-fantastic. In particular, with this strategy, either $G^{(T)}$ or $G^{(T-1)}$ is $k$-fantastic. Lemma \ref{lemma no long cycles} implies that either $G^{(T)}$ or $G^{(T-1)}$ contains no cycles of length $k$ or more, and so by Theorem \ref{thm:EG}, we must have
	\[
	|E(G^{(T)})|\leq \frac{(k-1)(n-1)}2+1.\qedhere
	\]
\end{proof}

\subsection{Near Optimality of the Conditions for $k$-dense Families}
Our work in the previous subsection shows that if $\c{C}$ is a family of $k$-dense cycles, then either player can play so that $G^{(t)}$ always has circumference less than $k$, which gives our desired bound on $\sat_g(n,\c{C})$ by Theorem~\ref{thm:EG}.  In this subsection we show that a slight loosening of the definition of $k$-fantastic families makes such a strategy impossible.  More precise, we prove the following. 
\begin{proposition}\label{prop:opt}
	Let $k\geq 5$. If $\mathcal{C}$ is a set of cycles such that there exists an $s\geq 3$ with $C_\ell \notin \mathcal{C}$ for all $s\leq \ell \leq 4 +(k-2)(s-2)+2(k-2)^2$, then for $n$ sufficiently large, either player can play the $\mathcal{C}$-saturation game such that the game ends with circumference at least $k$.
\end{proposition}
As a point of comparison, we remind the reader that for $k$-dense families and any $s\ge 3$, we have $C_\ell\in \c{C}$ for some
\[s+2\le \ell\le 3+(k-2)(s-2),\]
which is very close to the conditions of Proposition~\ref{prop:opt} when $s$ is large in terms of $k$.

To emphasize, Proposition~\ref{prop:opt} does not say that $\sat_g(n,\c{C})$ will not be linear for $\c{C}$ as in the proposition, only that any strategy in the $\c{C}$-saturation game which tries to guarantee that $G^{(t)}$ has circumference smaller than $k$ is doomed to fail.  Thus new ideas would be needed for computing $\sat_g(n,\c{C})$ for families of this form.

To prove Proposition~\ref{prop:opt}, we first establish the following technical lemma.

\begin{lem}\label{lemma density with given circumference}
	Let $G$ be a graph of circumference less than $k\ge 3$ and $P=(x_1,\ldots,x_{m})$ a path of maximum length in $G$.  For all $m'\le m$, we have that the longest path between $x_1$ and $x_{m'}$ has at most $m'-1+2(k-2)^2$ edges.
\end{lem}

\begin{proof}
	Let $Y$ be a longest path between $x_1$ and $x_{m'}$ and let $S=V(P)\cap V(Y)$.  

	If all the vertices in $S$ have indices less than or equal to $m'$, then $Y$ together with $(x_{m'+1},\ldots,x_{m})$ defines a path in $G$ from $x_1$ to $x_{m}$, which cannot be longer than $P$ since $P$ is a path of maximum length. Thus in this case we have $|E(Y)|\le m'-1$, which satisfies the condition of the lemma. 
	
	Thus we can assume there is some vertex in $S$ with index greater than $m '$.  Order the set $S$ based on the order that the vertices appear in $Y$, so that $x_1$ is the first vertex of $S$ and $x_{m'}$ the last.  Let $x_j$ be the smallest vertex of $S$ which has $j > m'$, and let $x_i$ be the vertex  which immediately precedes $x_j$ in $S$ based on the ordering of $S$. We claim that $x_i \in \{x_{m'-k+3}, \ldots, x_{m'-1}\}$. Indeed, $i<m'$ by definition of $j$, and if $i < m'-k+3$, then the portion of the path $Y$ between $x_i$ and $x_j$ together with the portion of the path $P$ between those two vertices forms a cycle of length at least $k$, a contradiction.

	Let $I = \{x_{m' - k + 3}, \ldots, x_{m' + k - 3} \}$ and let $V(Y) \cap I = \{x_{j_1}, \ldots , x_{j_p}\}$ with the vertices $x_{j_i}$ in the order in which they appear along the path $Y$. Notice in particular that $x_{j_p} = x_{m'}$, which by the previous claim implies $x_{j_1} \in \{x_{m'-k+3}, \ldots, x_{m'-1}\}$, and also that $V(Y) \cap I \neq \emptyset$. Decompose $Y$ into the subpaths $Y', Y''$ separated at $x_{j_1}$. By the argument in the previous paragraph, $Y'$ contains no vertices to the right of $x_{j_1}$ along $P$, so $Y'$ together with the path along $P$ from $x_{j_1}$ to $x_{m}$ defines a path in $G$ from $x_1$ to $x_{m}$, which cannot be longer than $P$. Hence $|Y'| \leq m'-1$. We will now show that $|Y''|\leq 2(k-2)^2$.

	For each $1 \leq i \leq p-1$, let $Y_i$ be the subpath of $Y''$ from $x_{j_i}$ to $x_{j_{i+1}}$. We claim that $|Y_i| \leq k-2$ for all $i$. Indeed, $Y_i$ is disjoint from the subpath $P_i$ in $P$ that goes from $x_{j_i}$ to $x_{j_{i+1}}$ because $Y_i$ only intersects $I$ at $x_{j_i}$ and $x_{j_{i+1}}$. Thus $P_i$ and $Y_i$ give rise to a cycle of length at least $|Y_i|+1$ in $G$. Since the circumference of $G$ is strictly less than $k$, $Y_i$ can have at most $k-2$ edges.
	
	Since $x_{j_1}, \ldots, x_{j_p} \in I$, we have $p \leq 2(k-3)+1$. Thus we have at most $2(k-3)$ paths $Y_i$, $1\leq p-1$. In total, $Y''$ can be decomposed into at most $2(k-3)$ many paths of length at most $k-2$, so $|Y''| \leq 2(k-2)^2$.  Combining this with the bound $|Y'|\le m'-1$ gives the desired result.

\end{proof}

With this we can now prove our main result for this subsection.

\begin{proof}[Proof of Proposition~\ref{prop:opt}]
Let $k'=3+(k-2)(s-2)$ with $s$ as in the hypothesis of the proposition. Initially we play as follows: if there is no path with at least $k'$ vertices in $G^{(t)}$ at the start of our turn, then we add an edge between an isolated vertex and the endpoint of some longest path in the graph.  If $n$ is sufficiently large, then eventually this will lead us to start our turn with some longest path in $G^{(t)}$ of length at least $k'$, call it $(x_1,\ldots,x_{m})$  with $m\ge k'$.  Once this is achieved, we subsequently attempt to add the edge $x_1x_{k'}$ to $G^{(t)}$.  If this is a legal move, then we are done since $k'\ge k$, and we have the desired result.

We can thus assume that adding the edge $x_{1}x_{k'}$ is not a legal move.  That is, $G^{(t)}$ must contain a path $Y=(y_1,y_2,\ldots,y_q)$ with $y_1=x_1,\ y_q=x_{k'}$ and $C_q\in \c{C}$.  Note by the previous lemma and the assumption $(x_1,\ldots,x_{m})$ a longest path implies $q\le k'-1+2(k-2)^2$, so by hypothesis of $\c{C}$ we must have $q<s$.
	
	Let $i_1$ be such that $x_{i_1}$ is the first $x_i$ vertex to appear in $Y$ (so $i_1=1$), and define $i_2,\ldots,i_p$ in an analogous same way.  This implies for all $1\le j<p$ that $G^{(t)}$ contains a cycle of length at least $i_j-i_{j-1}+1$.  Thus we can assume $i_j\le i_{j-1}+k-2$ for all $j$, as otherwise $G^{(t)}$ contains a cycle of length at least $k$.  By applying this bound repeatedly we find \[ 1+(k-2)(p-1)\ge i_p=k'=3+(k-2)(s-2).\]  Thus we must have $s\le p\le q$, contradicting that $q<s$ from above. Therefore it must be the case that $x_{1}x_{k'}$ is a legal move, allowing us to create a graph with circumference at least $k$.
\end{proof}

	\section{Concluding Remarks}\label{S:con}
	In this paper we considered the saturation game where Max made the first move, and one could instead consider the analogous game where Mini makes the first move.   It was shown by Hefetz, Krivelevich, Naor, and Stojakovi{\'c} \cite{HKNS} that in general these two saturation games can have dramatically different scores.  However, all of our proofs allowed for either player to implement the proposed strategies, and from this one can easily show that all of the bounds of our theorems continue to hold even if Mini makes the first move.

	Many of the results in this paper and in \cite{spiro} focused on families of odd cycles.  This is because in theory the game saturation number of a family of odd cycles could be anywhere between linear and quadratic.  Motivated by this, we ask the following.
	
	\begin{problem}
		Determine whether Theorem~\ref{thm:dense} continues to hold if $k$-dense families do not require $C_{k+1}\in \c{C}$ for $k$ odd.
	\end{problem}
	We have already shown that this is true for $k=5$, and we believe that with more work one can use our methods to show that this also holds for $k=7$ and possibly $k=9$, but beyond this new ideas are needed.

	In this paper we focused primarily on upper bounds for the game saturation number of a family of (odd) cycles.  It would be of interest to consider lower bounds as well.
	\begin{problem}
	Determine non-trivial asymptotic lower bounds on $\sat_g(n,\c{C}_k^o)$ for $k\ge 5$ odd.
	\end{problem}
	We can prove one such bound when $k=5$. By essentially the same argument used in the proof of Theorem~\ref{thm:fin}, one can show that Mini can maintain that $G^{(t)}$ has circumference at most 4 in the $\c{C}_5^o$-saturation game.  It was proven by Ferrara, Jacobson, Milans, Tennenhouse, and Wenger~\cite[Theorem~2.17]{Ferrara} that $\sat(n,\{C_5,C_6,C_7,\ldots\})=\f{10}{7}(n-1)$, which implies that $\sat_g(n,\c{C}_5^o)\ge \f{10}{7}(n-1)$.
	
	Lastly, we note that all of our examples of families of cycles $\c{C}$ with linear game saturation number contained infinitely many cycles.  It is unclear if every family of cycles with linear game saturation must have infinite size.
	
	\begin{problem}
		Determine whether $\sat_g(n,\c{C})=\om(n)$ whenever $\c{C}$ is a finite collection of cycles.
	\end{problem}
	
	\subsection*{Acknowledgments}
	
	The authors would like to thank Michael Ferrara, Florian Pfender, Cedar Wiseman, and Ted Tranel for their support and contributions during the initial discussions of this paper. This work was completed in part at the 2019 Graduate Research Workshop in Combinatorics, which was supported in part by NSF grant 1923238, NSA grant H98230-18-1-0017, a generous award from the Combinatorics Foundation, and Simons Foundation Collaboration Grants 426971 (to M. Ferrara), 316262 (to S. Hartke) and 315347 (to J. Martin).  
	
T.~Masa\v{r}\'ik received funding from the European Research Council (ERC) under the European Union’s Horizon 2020 research and innovation programme Grant Agreement 714704. He completed a part of this work while he was a postdoc at Simon Fraser University in Canada, where he was supported through NSERC grants R611450 and R611368. G. McCourt was supported in part by the National Science Foundation Research Training Group Grant DMS-19372.
	S.~Spiro was supported by the National Science Foundation Graduate Research Fellowship under Grant DGE-1650112.

	\bibliographystyle{plainurl}
	\bibliography{bib}
	
\end{document}